\documentclass[12pt]{amsart}

\usepackage{graphicx}
\usepackage{amsthm}
\usepackage{amssymb}

\newcommand{\R}{{\mathbb R}}

\theoremstyle{plain}
\numberwithin{equation}{section}
\newtheorem{thm}{Theorem}[section]
\newtheorem{theorem}[thm]{Theorem}
\newtheorem{lemma}[thm]{Lemma}

\newtheorem{prop}[thm]{Proposition}
\newtheorem{remark}[thm]{Remark}
\newtheorem{cor}[thm]{Corollary}

\newcommand{\eps}{\varepsilon}
\newcommand{\paren}[1]{\left(#1\right)}
\newcommand{\bracket}[1]{\left[#1\right]}

\newcommand{\inner}[1]{\left\langle#1\right\rangle}
\newcommand{\abs}[1]{\left\lvert#1\right\rvert}
\newcommand{\mb}{\mathbb}
\newcommand{\del}{\partial}
\newcommand{\pdif}[2]{\frac{\del#1}{\del#2}}

\begin{document}

\title[Isoperimetry in Surfaces of Revolution with Density]{Isoperimetry in Surfaces of Revolution with Density}

\thanks{This paper is a product of the work of the Williams College SMALL NSF REU 2017 Geometry Group.
We would like to thank our advisor Frank Morgan for his guidance on this project.
We would also like to thank the National Science Foundation; Williams College; Michigan State University; the University of Maryland, College Park; and the Massachusetts Institute of Technology.}
\author{Eliot Bongiovanni}
\address{Department of Mathematics\\
			Rice University}
\email{eliotbonge@gmail.com}
\author{Alejandro Diaz}
\address{Department of Mathematics\\
			The University of Maryland, College Park}
\email{adiaz126@terpmail.umd.edu}
\author{Arjun Kakkar}
\address{Department of Mathematics\\
                University of California, Los Angeles}
\email{arjunkakkar8@gmail.com}
\author{Nat Sothanaphan}
\address{Courant Institute of Mathematical Sciences\\
				New York University}
\email{natsothanaphan@gmail.com}

\begin{abstract}
The isoperimetric problem with a density or weighting seeks to enclose prescribed weighted volume with minimum weighted perimeter. According to Chambers' recent proof of the log-convex density conjecture, for many densities on $\mathbb{R}^n$ the answer is a sphere about the origin. We seek to generalize his results to some other spaces of revolution or to two different densities for volume and perimeter.
We provide general results on existence and boundedness and a new approach to proving circles about the origin isoperimetric.
\end{abstract}

\maketitle

\section{Introduction}
The log-convex density theorem proved by Gregory Chambers \cite{Ch}
asserts that on $\R^n$ with log-convex density, an isoperimetric surface is a sphere centered at the origin. We seek to generalize his results to some other spaces of revolution and to two different densities for volume and perimeter.

Our Theorems \ref{thm:exist} and \ref{thm:bounded} provide general results on existence and boundedness after Morgan and Pratelli \cite{MP}.
The existence proof shows that in the limit no volume is lost to infinity. The boundedness proof uses comparisons to derive a differential equation on volume growth.

Sections \ref{sec:congencurve} and \ref{sec:logconvexfh} focus on 2-dimensional surfaces of revolution with perimeter density and volume density equal. Our main Theorem \ref{MainTheorem} shows under the assumption that the product of the density and the metric factor is eventually log-convex
that, for large volumes, if the component farthest from the origin contains the origin, then an isoperimetric curve is a circle centered at the origin:

\begin{theorem}[Corollary \ref{cor:largevolume}]
\label{MainTheorem}
Consider $\R^2$ in polar coordinates $(r,\theta)$ with metric
$$ds^2=dr^2+h(r)^2d\theta^2$$
and radial density $f(r)$.
Suppose that $fh$ has positive derivatives and is eventually log-convex and
$(\log fh)'$ diverges to infinity.
Then, for large volumes, if the origin is interior to
the component farthest from the origin, an isoperimetric curve is a circle centered at the origin.
\end{theorem}

The idea of the proof, aided by Figure \ref{fig:expofig}, is as follows. We first show that if an isoperimetric curve is not
a circle centered at the origin, then it must go near the origin (Prop. \ref{prop:fheventuallylogconvex}).
By using estimates on the generalized curvature formula,
we prove that in the region where $fh$ is log-convex and nondecreasing,
the angle $\alpha$ from the radial vector to the tangent vector at
each point of the isoperimetric curve increases (Lemma \ref{lem:alphaincrease})
at an accelerating rate (Lemma \ref{lem:alphaaccel}).
Then we observe that in order for the isoperimetric curve to go near
the origin, it must travel a long distance, and the angle $\alpha$
would have to increase too much by what we have shown.
Putting these estimates together gives a contradiction
(Thm. \ref{thm:gammacirclelargevol}).

\begin{figure}[h]
\label{fig:expofig}
\centering
\includegraphics[width = \linewidth]{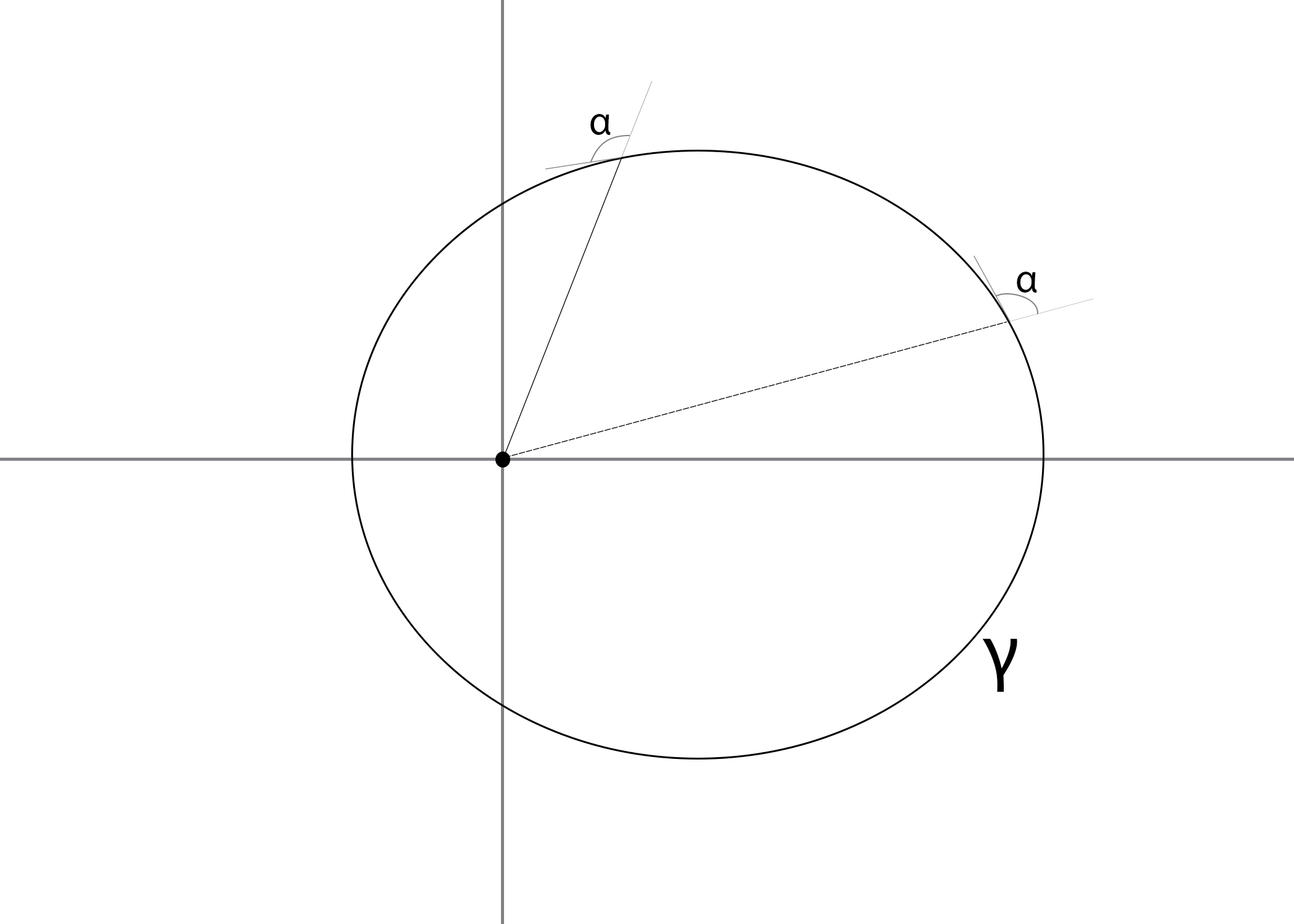}
\caption{The angle $\alpha$ from the radial vector to the tangent vector along the isoperimetric curve $\gamma$ increases at an accelerating rate, leading to a contradiction.}
\end{figure}

Similar results have been proven by Kolesnikov and Zhdanov \cite{KZ} for $\R^n$
with Euclidean metric
and Howe \cite{Ho} for a warped product of an interval with a Riemannian manifold,
without assuming that the component of an isoperimetric region
farthest from the origin contains the origin.
For details see Remark \ref{rem:otherresults}.

Finally, unless otherwise specified, when we mention perimeter and volume, we mean perimeter and volume weighted by the density. We also adopt the convention that $A\lesssim B$ if there is some positive dimension-dependent constant $c_n$ such that $A\leq c_n B$.

\section{Coordinates and First Variation}
\label{sec:prelim}
Let $H$ denote $\R^n$ in polar coordinates $(r,\Theta)$ with metric
$$ds^2 = dr^2 + h(r)^2d\Theta^2$$
and radial density $f(r)=e^{\psi(r)}$.
Define
$$B(r) := \{x \in H :|x| \leq r \}$$
as the ball of radius $r$.
For any region $E\subseteq H$, let $\abs{E}$ denote the measure of $E$.

\bigskip
The following first variation formula tells how perimeter varies as a region is deformed:

\begin{prop}
\label{prop:FirstVariation}
\emph{(First Variation Formula \cite [3.1, 3.2]{RCBM}).}
Let $f=e^\psi$ be a $C^1$ density on $H$. Then the initial first derivatives of volume and perimeter of a $C^1$ region $E$ with boundary $\partial E$, outward unit normal $\nu$, and inward mean curvature $H_0$, moving each boundary point $x$ with continuous normal velocity $u(x)$, are given by
$$V' = \int_{\partial E} u,\quad P' = \int_{\partial E} u H_f(x),$$
where
$$H_f(x)=H_0(x)+\pdif{\psi}{\nu}(x).$$
Consequently, for a smooth isoperimetric region, $H_f(x)$ is constant.
\end{prop}

\noindent The quantity $H_f(x)$ is called the \emph{generalized mean curvature}. By Proposition \ref{prop:FirstVariation}, it is the change of perimeter with respect to change in volume. (We are using the convention that the mean curvature is the sum rather than the mean of the principal curvatures.)

\section{Existence and Boundedness}
Existence and boundedness of isoperimetric regions for a single density for volume and perimeter on $\R^n$ were treated by Morgan and Pratelli \cite{MP}. Separate radial densities for volume and perimeter were treated by Di Giosia et al. \cite{DHKPZ}. We further allow certain radial metrics. More recently, other metrics have been treated by Pratelli and Saracco \cite{PS}.

To prove existence of isoperimetric regions, we begin with a lemma that puts a bound on the perimeter by using projection onto spheres.

\label{sec:existbounded}
\begin{lemma} 
\label{lem:projbound}
Let $H$ be $\R^n$ with metric
$$ds^2 = dr^2+h(r)^2d\Theta^2$$
and with nondecreasing density $f(r)$,
where $fh$ is nondecreasing.
If $E\subset H$ has finite volume, then for all $r > 0$,
$$\abs{\del E\setminus B(r)} \geq S(r),$$
where $S(r)$ is the area of the section of a sphere of radius $r$ sliced by the region $E$, $E\cap \del B(r)$.
\end{lemma}

\begin{proof}
The idea of the proof is to use projection.
Let $\pi: \del E \setminus B(r) \to \del B(r)$ be the radial projection
of the boundary of $E$ outside of $B(r)$ onto the sphere $\partial B(r)$.
Since $fh$ is nondecreasing, $\pi$ is measure nonincreasing:
$$\abs{\pi\paren{\del E \setminus B(r)}} \leq \abs{\del E\setminus B(r)}.$$
It remains to show that the left-hand side is greater than or equal
to $S(r)$.
For this, it is sufficient to show that
$\pi\paren{\del E \setminus B(r)}$ covers $E\cap \partial B(r)$
up to a set of measure zero.

Suppose the contrary. Then there exists a subset $X\subseteq E \cap  \partial B(r)$
of positive measure that is disjoint from
$\pi\paren{\del E \setminus B(r)}$.
So the product of $(r,\infty) \times X$ in polar coordinates
must be disjoint from the boundary $\del E$.
Since $X\subseteq E$, we must also have that $(r,\infty) \times X$
is contained in $E$.  But this would imply that $\abs{E}$ is infinite
(because $X$ has positive measure and $fh$ is nondecreasing),
which is a contradiction.
\end{proof}

The following theorem shows the existence of isoperimetric regions
by generalizing arguments of Morgan and Pratelli \cite{MP}.

\begin{theorem}
\label{thm:exist}
Let $H$ be $\mathbb{R}^n$ with metric
$$ds^2 = dr^2 + h(r)^2d\Theta^2,$$
volume density $f(r)$, and perimeter density $g(r)$.
Suppose that $h$ is nondecreasing, $g$ diverges to infinity,
and $f\leq cg$ for some constant $c$. Then an isoperimetric region exists for every positive volume less than the volume of the space $H$.
\end{theorem}

\begin{proof} The proof closely follows Morgan and Pratelli
\cite[Thm. 3.3]{MP}. The idea is to take a sequence of sets
with perimeters converging to the infimum  and extract a convergent subsequence.
The concern is that in the limit, some volume may be lost to infinity.
We suppose that there is some volume lost to infinity and show that it contradicts our assumption that perimeter density diverges to infinity.

Let $V$ be the prescribed (weighted) volume.
Consider a sequence of
smooth sets $E_j$  of volume $V$ and
$\abs{\del E_j}$ converges to the infimum. By compactness,
we may assume the sequence converges to a limit set $E$.

Suppose that some volume is lost to infinity. Then there exists $\eps > 0$ such that, for all $R>0$,
\begin{equation}
\label{eq:exist-lostvol}
\abs{E_j\setminus B(R)} \geq \eps
\end{equation}
for all $j$ large enough. Inequality \eqref{eq:exist-lostvol} then becomes
$$\int_R^\infty S_{j}(r) f(r)\, dr \geq \eps,$$
where $S_{j}(r)$ is the unweighted area under the metric $ds$ of the slice of $E_j$ by the sphere of radius $r$.
Define 
$$M_j:=\sup_{r\geq R} S_{j}(r), \quad g_- := \min_{r\geq R} g(r).$$
Notice that $g_-$ exists because $g$ is continuous and diverges to infinity.
Then by Lemma \ref{lem:projbound} (for unweighted volume) we have, for all $r\geq R$,
$$\abs{\del E_j} \geq \abs{\del E_j \setminus B(r)}
\geq \abs{\del E_j \setminus B(r)}_0 g_- \geq S_{j}(r) g_-,$$
where the subscript $0$ denotes the unweighted version. Therefore,
\begin{equation}
\label{eq:exist-firstineq}
\abs{\del E_j} \geq M_j g_-.
\end{equation}

For large $R$, since $g$ diverges,
$M_j$ is small (uniformly for all $j$),
and hence $S_j(r)$ is small.
Thus by the isoperimetric inequality on a sphere,
for all $r\geq R$,
$$p_{j}(r) \gtrsim S_{j}(r)^{\frac{n-2}{n-1}},$$
where $p_j(r)$ is the unweighted perimeter of the slice
of $E_j$ by the sphere of radius $r$.
Therefore, by the coarea formula,
\begin{align}
\label{eq:exist-bigineq}
\abs{\del E_j} &\geq \int_R^\infty p_{j}(r) g(r) \,dr 
\geq \frac{1}{c}\int_R^\infty p_{j}(r)f(r) \,dr \nonumber \\
&\gtrsim \int_R^\infty S_{j}(r)^{\frac{n-2}{n-1}} f(r) \,dr \nonumber\\
&\geq \frac{1}{({M_j})^{\frac{1}{n-1}}}
\int_R^\infty S_{j}(r) f(r) \,dr \nonumber \\
&\geq \frac{1}{(M_j)^{\frac{1}{n-1}}} \eps.
\end{align}
By \eqref{eq:exist-firstineq} and \eqref{eq:exist-bigineq}, 
$$\abs{\del E_j}^\frac{n}{n-1} \gtrsim
\eps g_-^\frac{1}{n-1}.$$
Since the left-hand side is uniformly bounded, $g_-$ is bounded
independent of $R$.
This contradicts that assumption that $g$ goes to infinity.

Therefore, there is no volume lost to infinity, $E$ has the prescribed volume and realizes the infimum perimeter.
\end{proof}

\begin{remark}
The argument used in Theorem \ref{thm:exist}
can be used to prove the existence of a perimeter-minimizing $n$-bubble
for any $n$ given volumes.
This can be shown by considering a sequence of $n$-bubbles with prescribed volumes whose perimeters tend towards the infimum.
If some volume is lost to infinity in the limit of the sequence,
then the same argument shows that the bubbles in the sequence
have perimeters going to infinity, which cannot be the case.
\end{remark}

Finally, by again generalizing arguments of Morgan and Pratelli \cite{MP}, we prove
boundedness of isoperimetric regions.

\begin{theorem}
\label{thm:bounded}
Let $H$ be $\mathbb{R}^n$ with metric
$$ds^2 = dr^2 + h(r)^2d\Theta^2,$$
volume density $f(r)$, and perimeter density $g(r)$.
Suppose that $gh$ is nondecreasing, $g^{n/(n-1)}/f$ is nondecreasing,
and $\int_0^\infty f^{1/n}$ diverges.
Then every isoperimetric region is bounded.
\end{theorem}

\begin{proof}
This proof closely follows Morgan-Pratelli \cite[Thm. 5.9]{MP}.  We begin by supposing that an isoperimetric region is unbounded.  Then from the isoperimetric inequality and the coarea formaula we derive that the volume of the region outside the ball of radius $r$ decreases uniformly
and hence becomes negative as $r$ increases, which is a contradiction.

Suppose that an isoperimetric region $E$ is unbounded.  Define
$$E_r:= E \cap \del B(r),$$
$$P(r) := \abs{\del E \setminus B(r)}_g,
\quad V(r):= \abs{E\setminus B(r)}_f,$$
where the subscript denotes the density for the measure.
By Lemma \ref{lem:projbound} for density $g$,
since $gh$ is nondecreasing,
\begin{equation}
\label{eq:bound-slicebound}
P(r) \geq \abs{E_r}_g.
\end{equation}
For $r$ large, $P(r)$ is small
and therefore $\abs{E_r}_g$ is small,
while the $g$-weighted volume of the sphere of radius $r$ is not small because
$gh$ is nondecreasing.
So the isoperimetric inequality on a sphere applies:
$$\abs{\del E_r}_0 \gtrsim \abs{E_r}^{\frac{n-2}{n-1}}_0,$$
where the subscript 0 indicates unweighted measure.
Multiplying both sides by the density $g(r)$ yields
\begin{equation}
\label{eq:bound-isoineq}
\abs{\del E_r}_g \gtrsim g(r)^{\frac{1}{n-1}}
\abs{E_r}^{\frac{n-2}{n-1}}_g.
\end{equation}
Inequalities \eqref{eq:bound-slicebound} and \eqref{eq:bound-isoineq}
then imply that
\begin{equation}
\label{eq:bound-combinedineq}
\abs{\del E_r}_g \gtrsim g(r)^{\frac{1}{n-1}}
P(r)^{-\frac{1}{n-1}}\abs{E_r}_g.
\end{equation}

Using the coarea formula \cite[\textsection 4.11]{Mo}, we can say that,
\begin{equation}
\label{eq:bound-sliceperi}
-P'(r) \geq \abs{\del E_r}_g.
\end{equation}
Meanwhile
\begin{equation}
\label{eq:bound-slicevol}
- V'(r)= \abs{E_r}_f.
\end{equation}
By inequalities \eqref{eq:bound-combinedineq}, \eqref{eq:bound-sliceperi},
and \eqref{eq:bound-slicevol},
\begin{align*}
-P'(r) &\gtrsim g(r)^{\frac{1}{n-1}} P(r)^{-\frac{1}{n-1}}\abs{E_r}_g \\
&= - \frac{g(r)^\frac{n}{n-1}}{f(r)} P(r)^{-\frac{1}{n-1}} V'(r),
\end{align*}
which simplifies to
$$-\frac{d}{dr} \paren{P(r)^{\frac{n}{n-1}}} \gtrsim
- \frac{g(r)^\frac{n}{n-1}}{f(r)} \frac{d}{dr} V(r),$$
where $c_n$ is a new dimensional constant.
Since $E$ has finite perimeter and volume,
$P(r)$ and $V(r)$ both go to zero as $r$ goes to infinity.
Hence integration of both sides of the previous inequality yields 
\begin{align}
P(r)^{\frac{n}{n-1}} &\gtrsim
- \int_r^\infty \frac{g(t)^\frac{n}{n-1}}{f(t)} \frac{d}{dt} V(t)\,dt
\nonumber \\
&\geq - \frac{g(r)^\frac{n}{n-1}}{f(r)} \int_r^\infty \frac{d}{dt} V(t)\,dt
\nonumber \\
&= \frac{g(r)^\frac{n}{n-1}}{f(r)} V(r),
\label{eq:bound-ptov}
\end{align}
because $g^{n/(n-1)}/f$ is nondecreasing
(and the right-hand side is positive).

Choose $R$ so that the interior of the ball of radius $R$
contains part of the boundary of $E$.
Then, for sufficiently small $\eps>0$, we can define a set $E_\eps$
by introducing a variation on the boundary of $E$ inside $B(R)$
to increase the weighted volume by $\eps$.
Since the (constant) generalized mean curvature $H(E)$ is $dP/dV$
(Prop. \ref{prop:FirstVariation}) we have
$$\lim_{\eps \to 0}\frac{|\del E_\eps|_g-|\del E|_g}{\eps}=H(E).$$

Therefore, for small $\eps$,
\begin{equation}
\label{eq:bound-HE}
\abs{\del E_\eps}_g \leq \abs{\del E}_g + \eps \paren{H(E) + 1}.
\end{equation}
Take $r>R$ large enough
such that $\eps=V(r)$ is small enough for this construction.
If we replace $E_\eps$ by $\widetilde{E}:= E_\eps \cap B(r)$,
discarding the volume $V(r)$, then $\widetilde{E}$
is back to the original volume of $E$.
Since $E$ is isoperimetric,
\begin{equation}
\label{eq:bound-isoperimetric}
|\del\widetilde{E}|_g \geq \abs{\del E}_g.
\end{equation}
On the other hand, since $\widetilde{E}$ loses the perimeter $P(r)$
outside the ball and creates new perimeter $E_r$,
it follows that
\begin{align}
|\del\widetilde{E}|_g &= \abs{\del E_\eps}_g - P(r) + \abs{E_r}_g \nonumber \\
&\leq \abs{\del E}_g + \eps\paren{H(E)+1} -
c_n\frac{g(r)}{f(r)^{\frac{n-1}{n}}}\eps^{\frac{n-1}{n}} + \abs{E_r}_g
\label{eq:bound-bigineq}
\end{align}
by inequalities \eqref{eq:bound-HE} and \eqref{eq:bound-ptov}, where $c_n$ is a dimension-dependent constant.

For $r$ large, $\eps$ is small, and so
$\eps^{\frac{n-1}{n}}$ asymptotically dominates $\eps$.
From \eqref{eq:bound-isoperimetric} and \eqref{eq:bound-bigineq},
\begin{equation}
\label{eq:bound-ertovr}
\abs{E_r}_g \gtrsim \frac{g(r)}{f(r)^{\frac{n-1}{n}}}\eps^{\frac{n-1}{n}}
= \frac{g(r)}{f(r)^{\frac{n-1}{n}}}V(r)^{\frac{n-1}{n}}.
\end{equation}

Note that
$$\abs{E_r}_f = \abs{E_r}_g \frac{f(r)}{g(r)}.$$
Therefore by \eqref{eq:bound-slicevol} and \eqref{eq:bound-ertovr},
for $r$ sufficiently large,
$$-V'(r) \gtrsim f(r)^{\frac{1}{n}}V(r)^{\frac{n-1}{n}},$$
which is equivalent to 
$$\frac{d}{d r} \paren{V(r)^{\frac{1}{n}}}
\lesssim - f(r)^{\frac{1}{n}}.$$
Integrating both sides and using the fact that
$\int_0^\infty f^{1/n}$ diverges, we find that
$V(r)\to-\infty$ as $r\to\infty$, which is a contradiction.
Therefore $E$ is bounded. 
\end{proof}

\section{Constant Generalized Curvature Curves in 2D}
\label{sec:congencurve}
In this section, we consider the 2D case,
$\R^2$ in polar coordinates $(r,\theta)$ with metric
$$ds^2 = dr^2 + h(r)^2d\theta^2$$
and radial density $f(r)=e^{\psi(r)}$.
Following Chambers \cite[Sect. 2]{Ch},
let $A$ be an isoperimetric set spherically symmetrized.
Let $\gamma:[-\beta,\beta]\to \R^2$ be the arclength paramaterization
of the most distant component of the boundary of $A$ from the
leftmost point on the $x$-axis back to itself, counterclockwise.
Then $\gamma$ is symmetric about the $x$-axis,
$\gamma(0)$ and $\gamma(\pm\beta)$ are on the $x$-axis,
$\gamma$ is above the $x$-axis on $(0,\beta)$,
and $\gamma$ is below the $x$-axis on $(-\beta,0)$.
By known regularity \cite{Mo}, $\gamma$ is a smooth curve.

Let $\hat{r}(t)$ and $\hat{\theta}(t)$ be the orthonormal basis vectors
of the tangent space at $\gamma(t)$ in the radial and tangential directions
(unless $\gamma(t)$ is the origin).
Let $\alpha(t)$ be the counter-clockwise angle measured from $\hat{r}(t)$
to $\gamma'(t)$ at $\gamma(t)$. Note that the angles are measured with respect
to the defined metric and not the standard metric in $\mathbb{R}^2$. 

Observe that
$$\gamma'=r'\hat{r}+h(r)\theta'\hat{\theta},$$
\begin{equation}
\label{eqn:rthetaalpha}
r'=\cos\alpha,\quad h(r)\theta'=\sin\alpha.
\end{equation}
Let $\kappa(t)$ be the inward (leftward) curvature of $\gamma$ at $\gamma(t)$.
The generalized curvature is
$$\kappa_f(t)=\kappa(t)+\pdif{\psi}{\nu},$$
where $\nu$ is the unit outward normal at $\gamma(t)$.
Recall that $f=e^\psi$.
By the first variation formula (Prop. \ref{prop:FirstVariation}) and the fact that $A$ is an isoperimetric region,
$\kappa_f(t)$ is constant for all $t$.

We seek to analyze the constant generalized curvature curve $\gamma$.
First we need an explicit formula for the curvature.

\begin{lemma}
\label{lem:curvaturepolar}
The curvature of $\gamma$ at $t$ is
$$\kappa(t)=h(r)^2h'(r)\theta'^3+2h'(r)r'^2\theta'
+h(r)\paren{r'\theta''-\theta'r''},$$
where the polar coordinates $(r,\theta)$ of $\gamma$
are functions of $t$.
\end{lemma}

\begin{proof}
In polar coordinates, $\R^2$ with the given metric has first fundamental form
$$\begin{pmatrix} E & F \\ F & G \end{pmatrix}
= \begin{pmatrix} 1 & 0 \\ 0 & h(r)^2 \end{pmatrix}.$$
The curvature of $\gamma$ at $t$ is the geodesic curvature,
which is given by
\begin{multline*}
\kappa(t)=
\sqrt{EG-F^2}
\big[\Gamma^2_{11}r'^3-\Gamma^1_{22}\theta'^3+
\paren{2\Gamma^2_{12}-\Gamma^1_{11}}r'^2\theta'\\
-\paren{2\Gamma^1_{12}-\Gamma^2_{22}}r'\theta'^2
-r''\theta'+\theta''r'\big]
\big/\paren{Er'^2+2Fr'\theta'+G\theta'^2}^{3/2},
\end{multline*}
where $\Gamma^k_{ij}$ are the Christoffel symbols of the second kind.
Since $F=0$,
\begin{align*}
\Gamma^1_{11}&=\frac{E_r}{2E}=0,&
\Gamma^1_{12}&=\frac{E_\theta}{2E}=0,&
\Gamma^1_{22}&=-\frac{G_r}{2E}=-h(r)h'(r), \\
\Gamma^2_{11}&=-\frac{E_\theta}{2G}=0,&
\Gamma^2_{12}&=\frac{G_r}{2G}=\frac{h'(r)}{h(r)},&
\Gamma^2_{22}&=\frac{G_\theta}{2G}=0.
\end{align*}
Therefore
$$\kappa(t)=
\bracket{h(r)^2h'(r)\theta'^3+2h'(r)r'^2\theta'
+h(r)(r'\theta''-\theta'r'')}\big/\paren{r'^2+h(r)^2\theta'^2}^{3/2}.$$
The denominator is 1 due to arclength parametrization,
implying the desired formula.
\end{proof}

By using $\alpha$ (the angle from $\hat{r}(t)$ to $\gamma'(t)$), the curvature formula
can be further simplified:

\begin{prop}
\label{prop:curvaturealpha}
The curvature of $\gamma$ at $t$ is
$$\kappa(t)=\frac{h'(r)}{h(r)}\sin\alpha + \alpha'.$$
\end{prop}

\begin{proof}
Recall from (\ref{eqn:rthetaalpha}) that $r'=\cos\alpha$ and $\theta'=\sin\alpha/h(r)$.
The desired formula follows from Lemma \ref{lem:curvaturepolar}
by direct computation.
\end{proof}

The generalized curvature can now be explicitly computed.

\begin{prop}
\label{prop:gencurvatureformula}
The generalized curvature of $\gamma$ at $t$ is
$$\kappa_f(t)=(\log fh)'(r)\sin\alpha + \alpha'.$$
\end{prop}

Note that $f$ and $h$ are functions of $r$
but $\alpha$ is a function of $t$.

\begin{proof}
By Proposition \ref{prop:curvaturealpha} and the definition of
generalized curvature, it suffices to prove that
$$\pdif{\psi}{v}=\frac{f'(r)}{f(r)}\sin\alpha=\psi'(r)\sin\alpha.$$
The gradient of $\psi$ is
\begin{align*}
\nabla \psi &= \pdif{\psi}{r}\hat{r}+\frac{1}{h(r)}\pdif{\psi}{\theta}\hat{\theta}
=\psi'(r)\hat{r}
\end{align*}
because $\psi$ is radial. The unit outward normal is
\begin{align*}
\nu &= h(r)\theta'\hat{r}-r'\hat{\theta}.
\end{align*}
Hence
$$\pdif{\psi}{v}=\inner{\nabla \psi,\nu}=\psi'(r)h(r)\theta'=\psi'(r)\sin\alpha,$$
as asserted.
\end{proof}

From spherical symmetrization, some properties of $\alpha$ can be deduced.

\begin{lemma}
\label{lem:alphaproperty}
Assuming $\gamma$ avoids the origin, the angle $\alpha$ satisfies 
\begin{align*}
\alpha(0) &= \pi/2, \\
\alpha(-\beta) &= \alpha(\beta)=\pi/2 \text{ or } 3\pi/2, \text{ and } \\
\pi/2 & \leq \alpha(t) \leq 3\pi/2,
\end{align*} 
for all $t\in [0,\beta]$.
\end{lemma}

\begin{proof}
From spherical symmetrization, $\cos\alpha=r'(t)\leq 0$ for all $t\in [0,\beta]$,
implying the third assertion. Because $r(0)$ is maximum,
$\cos\alpha(0)=r'(0)=0$. So $\alpha(0)=\pi/2$ because $\gamma$ has
counter-clockwise parametrization. The second assertion follows from
the fact that $r(\beta)$ is minimum, so $\cos\alpha(\beta)=r'(\beta)=0$.
\end{proof}

\begin{remark}
\label{rem:anycomponent}
The results of this section hold for any component of an
isoperimetric region, not only for the farthest component $\gamma$.
Moreover, by Proposition \ref{prop:FirstVariation},
the generalized curvature (Prop. \ref{prop:gencurvatureformula})
of each component has to be equal.
\end{remark}

\section{Circles Isoperimetric}
\label{sec:logconvexfh}
In this section, with the assumption that the product $fh$ of the density and the metric factor is eventually log-convex, we will prove that
for large volume, an isoperimetric curve whose farthest component $\gamma$ encloses the origin is a circle centered at the origin.
The notation carries over from Section \ref{sec:congencurve}. In particular, $\alpha(t)$ denotes the angle from the radial to the tangent at $\gamma(t)$.
First we need a lemma.

\begin{lemma}
\label{lem:gamma'0}
If $\alpha'(0)=0$, then $\gamma$ is a circle centered at the origin.
\end{lemma}

\begin{proof}
Notice that a circle centered at the origin satisfies the
constant generalized curvature equation (Prop. \ref{prop:gencurvatureformula})
and has $\alpha(0)=\pi/2$
and $\alpha'(0)=0$. Therefore,
by the uniqueness of solutions of ODEs, $\gamma$ is a circle
centered at the origin.
\end{proof}

The next lemma shows that the fact that $\gamma$ is a circle
about the origin is enough to conclude
that an isoperimetric curve has only one component.

\begin{lemma}
\label{lem:onecomponent}
Suppose that $fh$ has positive derivatives.
If the farthest component $\gamma$ of an isoperimetric curve
is a circle centered at the origin, then the whole isoperimetric curve
is that circle centered at the origin.
\end{lemma}

\begin{proof}
The isoperimetric curve cannot have other components outside of
its farthest component $\gamma$
because $\gamma$ is a circle about the origin.
Suppose that there are other components inside $\gamma$; then some component $\bar{\gamma}$ must have clockwise orientation. 
By Proposition \ref{prop:gencurvatureformula} and the hypothesis on $fh$, $\gamma$ has positive generalized curvature. 
Similarly, by Remark \ref{rem:anycomponent}, the oppositely-oriented $\bar{\gamma}$ has negative generalized curvature. This contradicts the fact that an isoperimetric curve has constant generalized curvature (Prop. \ref{prop:FirstVariation}).
Therefore $\gamma$ is the whole isoperimetric curve.
\end{proof}

The following proposition shows that if $fh$ is eventually log-convex
and $\gamma$ is not a circle
centered at the origin, then it must go near the origin when it crosses the $x$-axis at $r(\beta)$. Recall that $r(t)$ is the distance from the origin to $\gamma(t)$.

\begin{prop}
\label{prop:fheventuallylogconvex}
If $fh$ is log-convex on the interval $[r_0,\infty)$
and the origin is interior to $\gamma$,
then either $\gamma$ is a circle centered at the origin or $r(\beta)< r_0$.
\end{prop}

\begin{proof}
Suppose that $r(\beta)\geq r_0$. We must show that $\gamma$ is a circle
centered at the origin.
Since $\gamma$ encloses the origin, Lemma \ref{lem:alphaproperty} applies, $\alpha(\beta)=\pi/2$, and $\pi/2 \leq \alpha(t) \leq 3\pi/2$ for all $t\in[0,\beta]$. Hence
$\alpha'(0)\geq 0$ and $\alpha'(\beta)\leq 0$.
The generalized curvature formula (Prop. \ref{prop:gencurvatureformula})
implies that
$$(\log fh)'(r(0))+\alpha'(0)=(\log fh)'(r(\beta))+\alpha'(\beta).$$
By spherical symmetrization, $r(0)\geq r(\beta)\geq r_0$, so by log-convexity of $fh$,
$$(\log fh)'(r(0))\geq (\log fh)'(r(\beta)).$$
This implies that $\alpha'(0)\leq \alpha'(\beta)$, so that $\alpha'(0)=\alpha'(\beta)=0$.
Hence by Lemma \ref{lem:gamma'0}, $\gamma$ is a circle
centered at the origin.
\end{proof}

\begin{lemma}
\label{lem:alphaleqpi}
If $fh$ is nondecreasing at $r(0)$
and the origin is interior to $\gamma$, then
$\alpha(t)\in[\pi/2,\pi]$ for all $t\in[0,\beta]$.
\end{lemma}

\begin{proof}
Suppose to the contrary that $\alpha(t)>\pi$ for some
$t\in[0,\beta]$. By Lemma \ref{lem:alphaproperty},
$\alpha(0)=\pi/2$ and, because $\gamma$ encloses the origin, $\alpha(\beta)=\pi/2$.
Thus there are $t_0<t<t_1$ such that
$\alpha(t_0)=\alpha(t_1)=\pi$, $\alpha'(t_0)\geq 0$, and $\alpha'(t_1) \leq 0$.
If $\gamma$ has constant generalized curvature $c$, then by the generalized curvature
formula (Prop. \ref{prop:gencurvatureformula})
$$c=\alpha'(t_0)=\alpha'(t_1),$$
so all three quantities have to be zero. So at $t=0$,
$$0=(\log fh)'(r(0))+\alpha'(0).$$
The first term on the right-hand side is nonnegative by hypothesis,
and the second term is nonnegative because $\alpha'(0)\geq 0$ (Lemma \ref{lem:alphaproperty}).
Therefore $\alpha'(0)=0$.
By Lemma \ref{lem:gamma'0}, $\gamma$ is a circle centered at the origin,
and $\alpha(t)=\pi/2$ for all $t\in[0,\beta]$, a contradiction.
Hence $\alpha(t)\leq \pi$ for all $t\in[0,\beta]$.
\end{proof}

We now show that $\alpha$ is nondecreasing (Lemma \ref{lem:alphaincrease})
and that its rate of increase is accelerating (Lemma \ref{lem:alphaaccel})
in the region where $fh$ is log-convex.

\begin{lemma}
\label{lem:alphaincrease}
If $fh$ is nondecreasing and log-convex on the interval $[r_0,\infty)$,
and the origin is interior to $\gamma$, then for any $t\in[0,\beta]$ such that
$r(t)\geq r_0$, $\alpha'(t)\geq 0$.
\end{lemma}

\begin{proof}
Assume that $r(0)\geq r_0$, otherwise the statement is trivial.
By Lemma \ref{lem:alphaleqpi}, $\alpha(t)\in[\pi/2,\pi]$ for all $t\in[0,\beta]$.
By Lemma \ref{lem:alphaproperty}, $\alpha'(0)\geq 0$.
If $\alpha'(0)=0$, then Lemma \ref{lem:gamma'0} implies that $\gamma$
is a circle centered at the origin, and the lemma holds.
So suppose $\alpha'(0)>0$.
Assume for contradiction that there is a $t$ for which $r(t)\geq r_0$ and $\alpha'(t)< 0$.
Let $t_0>0$ be the smallest value of $t$ such that
such that $r(t)\geq r_0$ and $\alpha'(t)=0$.
For $t<t_0$, the generalized curvature
formula (Prop. \ref{prop:gencurvatureformula}) gives
$$(\log fh)'(r(t))\sin\alpha(t)+\alpha'(t)=(\log fh)'(r(t_0))\sin\alpha(t_0).$$
Because $\alpha'(t)>0$, it must be that
\begin{equation}
\label{eq:alphaincreasebd}
(\log fh)'(r(t))\sin\alpha(t)<(\log fh)'(r(t_0))\sin\alpha(t_0).
\end{equation}
Note that $\pi/2\leq\alpha(t)<\alpha(t_0)\leq\pi$ by construction,
so $\sin\alpha(t)>\sin\alpha(t_0)\geq 0$.
Moreover, because $r(t)\geq r(t_0)\geq r_0$, by hypothesis,
$$(\log fh)'(r(t))\geq(\log fh)'(r(t_0))\geq 0.$$
So the left-hand side of \eqref{eq:alphaincreasebd} is greater than or equal
to its right-hand side, a contradiction.
Therefore the lemma holds.
\end{proof}

\begin{lemma}
\label{lem:alphaaccel}
If $fh$ is nondecreasing and log-convex on the interval $[r_0,\infty)$
and the origin is interior to $\gamma$, then for any $t\in[0,\beta]$ such that
$r(t)\geq r_0$, $\alpha''(t)\geq 0$.
\end{lemma}

\begin{proof}
Fix $t\in[0,\beta]$ such that $r(t)\geq r_0$.
By Lemma \ref{lem:alphaincrease}, $\alpha'(t)\geq 0$, so $\alpha$ is nondecreasing.
By Lemma \ref{lem:alphaleqpi}, $\alpha(t)\in[\pi/2,\pi]$.
Recall the generalized curvature formula (Prop. \ref{prop:gencurvatureformula}):
$$\kappa_f(t)=(\log fh)'(r)\sin\alpha + \alpha'.$$
Because $(\log fh)'(r)$ is nonnegative and nonincreasing as a function of $t$ and
$\sin\alpha$ is nonnegative and nonincreasing, $\alpha'$ is nondecreasing.
Hence $\alpha''(t)\geq 0$.
\end{proof}

The next theorem proves the circle isoperimetric, replacing the hypothesis on close approach to the origin of Proposition \ref{prop:fheventuallylogconvex} with a lower bound $M$ on $(\log fh)'$ at the point farthest from the origin.

\begin{theorem}
\label{thm:gammacirclelargevol}
Consider $\R^2$ in polar coordinates $(r,\theta)$ with metric
$$ds^2=dr^2+h(r)^2d\theta^2$$
and radial density $f(r)$.
Suppose that $fh$ has positive derivatives and that,
on the interval $[r_0,\infty)$, it is log-convex.
Let
$$M=\inf_{r>r_0} \bracket{(\log fh)'(r)+\frac{\pi}{2(r-r_0)}}.$$
Suppose that the origin is interior to the component of an isoperimetric curve
farthest from the origin and
the farthest distance from the origin $r_\textnormal{max}$ satisfies
$$r_\textnormal{max}> r_0,\quad (\log fh)'(r_\textnormal{max})>M.$$
Then the isoperimetric curve is a circle centered at the origin.
\end{theorem}

\begin{proof}
The idea of the proof is that if the farthest component $\gamma$
goes near the origin, then it has
to travel a long distance to reach the region near the origin, and $\alpha$ would
have to increase too much. See Figure \ref{fig:expofig}. 

By Lemma \ref{lem:onecomponent}, it suffices to show that
$\gamma$ is a circle centered at the origin.
Suppose the contrary.
By Proposition \ref{prop:fheventuallylogconvex}, $r(\beta)< r_0$.
Since $r(0)>r_0$,  there is a $t_0$ such that $r(t_0)=r_0$.
By Lemmas \ref{lem:alphaincrease} and \ref{lem:alphaaccel},
$\alpha'(t)\geq 0$ and $\alpha''(t)\geq 0$ for all $t\in[0,t_0]$.
Since $(\log fh)'(r(0))>M$, there is an $r_1>r_0$ such that
$$(\log fh)'(r(0))>(\log fh)'(r_1)+\frac{\pi}{2(r_1-r_0)}.$$
Log-convexity of $fh$ implies that $r_1<r(0)$, so there is $t_1<t_0$ such that $r(t_1)=r_1$.
Because we are using arclength parametrization, it must be that
$t_0-t_1\geq r_1-r_0$.
Because $\alpha(t_1)\geq\pi/2$ and $\alpha(t_0)\leq\pi$ (Lemma \ref{lem:alphaleqpi}),
$$\frac{\pi}{2}\geq \alpha(t_0)-\alpha(t_1)=\int_{t_1}^{t_0} \alpha'\geq (t_0-t_1) \alpha'(t_1)
\geq (r_1-r_0)\alpha'(t_1),$$
so that
$$\alpha'(t_1)\leq \frac{\pi}{2(r_1-r_0)}.$$
By the constant generalized curvature formula (Prop. \ref{prop:gencurvatureformula}),
\begin{align*}
(\log fh)'(r(0))&\leq (\log fh)'(r(0))+\alpha'(0) \\
&=(\log fh)'(r_1)\sin\alpha(t_1)+\alpha'(t_1) \\
&\leq(\log fh)'(r_1)+\alpha'(t_1) \\
&\leq (\log fh)'(r_1)+\frac{\pi}{2(r_1-r_0)},
\end{align*}
a contradiction.
Therefore $\gamma$ is a circle centered at the origin.
By Lemma \ref{lem:onecomponent}, the whole isoperimetric curve
is that circle centered at the origin.
\end{proof}

The hypothesis of Theorem \ref{thm:gammacirclelargevol} can be satisfied
for large volumes whenever $fh$ is eventually log-convex and $(\log fh)'$ diverges to infinity,
as shown in the following corollary.

\begin{cor}
\label{cor:largevolume}
Suppose that $fh$ has positive derivatives and is eventually log-convex and
$(\log fh)'$ diverges to infinity.
Then, for large volumes, if the origin is interior to
the component farthest from the origin, an isoperimetric curve is a circle centered at the origin.
\end{cor}

\begin{proof}
Apply Theorem \ref{thm:gammacirclelargevol}.
For large volumes, $r(0)$ is large, so $(\log fh)'(r(0))>M$.
Hence $\gamma$ is a circle centered at the origin.
\end{proof}

\begin{remark}
\label{rem:otherresults}
Similar results to Corollary \ref{cor:largevolume} are proven by Kolesnikov and Zhdanov \cite{KZ} and Howe \cite{Ho}, without assuming that the component farthest from the origin of an isoperimetric region contains the origin.
Kolesnikov and Zhdanov use the divergence theorem to show
that isoperimetric surfaces in $\R^n$ for large volumes
are spheres about the origin \cite[Prop. 6.7]{KZ}.
Howe uses vertical area to prove that isoperimetric regions
in a warped product of an interval with a Riemannian manifold
for large volumes are vertical fibers \cite[Cor. 2.10]{Ho}.
\end{remark}

The following corollary applies Theorem \ref{thm:gammacirclelargevol}
to the example of the Borell density $e^{r^2}$ on the hyperbolic plane.

\begin{cor}
Consider the hyperbolic plane $\mb{H}^2$ with density
$e^{r^2}$. Let $r_0=\sinh^{-1}(1/\sqrt{2})$,
$M$ be as in Theorem \ref{thm:gammacirclelargevol}, $r^*>r_0$
be such that $(\log fh)'(r^*)=M$, and
$$V_0=2\pi(\cosh r^*-1) \approx 31.098.$$
Then for any volume larger than $V_0$,
if the origin is interior to the component farthest from the origin,
an isoperimetric curve is a circle centered at the origin.
\end{cor}

\begin{proof}
The product $fh=e^{r^2}\sinh r$ is log-convex and nondecreasing on
$[r_0,\infty)$, so we can apply Theorem \ref{thm:gammacirclelargevol}.
Since $V_0$ is the area of the hyperbolic circle with radius $r^*$,
for any volume larger than $V_0$, $r(0)>r^*>r_0$, so that $(\log fh)'(r(0))>M$.
Therefore, by Theorem \ref{thm:gammacirclelargevol},
$\gamma$ is a circle centered at the origin.
\end{proof}

\medskip

\noindent MSC2010: 51F99

\medskip

\noindent Key words and phrases: isoperimetric, surfaces of revolution, density


\begin{thebibliography}{99}

\bibitem[Ch]{Ch} Gregory R. Chambers. Proof of the Log-Convex Density Conjecture. J. Eur. Math. Soc., 2015, to appear.
\bibitem[DHKPZ]{DHKPZ} Leonardo Di Giosia, Jahangir Habib, Lea Kenigsberg, Dylanger Pittman, Weitao Zhu. Balls Isoperimetric in $\R^n$ with Volume and Perimeter Densities $r^m$ and $r^k$, arXiv:1610.05830, 2016.
\bibitem[KZ]{KZ} Alexander Kolesnikov, Roman Zhdanov. On isoperimetric sets of radially symmetric measures. In Christian Houdr\'{e}, Michel Ledoux, Emanuel Milman, and Mario Milman, editors, Concentration, Functional Inequalities and Isoperimetry (Proc. intl. wkshp., Florida Atlantic Univ., Oct./Nov. 2009), number 545 in Contemporary Mathematics, pages 123--154. Amer. Math. Soc., 2011.
\bibitem[Ho]{Ho} Sean Howe. The Log-Convex Density Conjecture and vertical surface area in warped products, Advances in Geometry, Vol. 15 Issue 4, 2015, 455--468.
\bibitem[Mo]{Mo} Frank Morgan. \textit{Geometric Measure Theory}. 5th ed., 2016.
\bibitem[MP]{MP} Frank Morgan, Aldo Pratelli. Existence of isoperimetric regions in $\R^n$ with density, Ann. Glob. Anal. Geom., Vol. 43, 2013, 331--365.
\bibitem[PS]{PS} Aldo Pratelli, Giorgio Saracco. On the isoperimetric problem with double density, Nonlinear Analysis, 10.1016/j.na.2018.04.009, 2018. 
\bibitem[RCBM]{RCBM} C\'esar Rosales, Antonio Ca\~nete, Vincent Bayle, Frank Morgan. {On the isoperimetric problem in Euclidean space with density}, Calc. Var. 31, 2008, 27-46.

\end{thebibliography}
\end{document}